\documentclass[twoside]{article}

\usepackage{blindtext} 

\usepackage[sc]{mathpazo} 
\usepackage[T1]{fontenc} 
\usepackage{microtype} 

\usepackage[english]{babel} 

\usepackage[hmarginratio=1:1,top=32mm,columnsep=20pt]{geometry} 
\usepackage[hang, small,labelfont=bf,up,textfont=it,up]{caption} 
\usepackage{booktabs} 

\usepackage{lettrine} 

\usepackage{enumitem} 
\setlist[itemize]{noitemsep} 

\usepackage{abstract} 

\usepackage{titlesec} 
\renewcommand\thesection{\Roman{section}} 
\renewcommand\thesubsection{\roman{subsection}} 
\titleformat{\section}[block]{\large\scshape\centering}{\thesection.}{1em}{} 
\titleformat{\subsection}[block]{\large}{\thesubsection.}{1em}{} 

\usepackage{fancyhdr} 
\usepackage{titling} 

\usepackage{hyperref} 

\pretitle{\begin{center}\Huge\bfseries} 
\posttitle{\end{center}} 
\title{Exponential stability of linear systems under a class of Desch-Schappacher perturbations}
\author{%
\textsc{S. El Alaoui and M. Ouzahra}\\
\normalsize M2PA Laboratory, University of
Sidi Mohamed Ben Abdellah \\
P.O. Box 5206, Bensouda, F\`es, Morocco\\ 
\normalsize \href{mohamed.ouzahra@usmba.ac.ma}{mohamed.ouzahra@usmba.ac.ma} }
\date{}
\renewcommand{\maketitlehookd}{%
\begin{abstract}
In this paper we investigate the uniform exponential stability of the  system $\frac{dx(t)}{dt}=Ax(t)-\rho Bx(t), \; (\rho >0), $ where the unbounded operator $A$ is the infinitesimal generator of a linear  $C_0-$semigroup of contractions $S(t)$ in a Hilbert space $X$ and $B$ is a Desch-Schappacher operator. Then we give sufficient conditions for exponential stability of the above system. The obtained stability result is  then applied
to show the uniform exponential stabilization of  bilinear partial differential equations.
\end{abstract}

{\bf Keywords}: Exponential stabilization, linear system,  bilinear control, Desch-Schappacher operator, unbounded control operator.}

\newtheorem{theorem}{Theorem}[section]

\newtheorem{lemma}[theorem]{Lemma}

\newtheorem{remark}{Remark}
\newtheorem{example}{Example}

\newtheorem{proof}{Proof}

\usepackage{amssymb}
\usepackage{amsmath}

\begin{document}

\maketitle

\section{Introduction}
\label{intro}
Consider the  following abstract  system
      \begin{eqnarray}\label{S1}
      \begin{cases}
      \dot{x}(t)=Ax(t)-\rho Bx(t),\;\;t>0\\
       x(0)=x_0,\\
      \end{cases}
      \end{eqnarray}
where the state $x(.)$  takes values in a Hilbert state space $X$ endowed with an  inner product $\left\langle.,.\right\rangle_X$ with associate norm $\|.\|_X$,  the unbounded operator $(A,D(A))$ generates  a  $C_0-$semigroup $S(t)$ on $X.$  Here,  $B$ is an unbounded linear operator of $X$  in the sense that it is bounded from $X$ to some extrapolating space of $X$. In the case of various real problems,  the modeling may lead to mathematical model  of the form (\ref{S1}) with an operator $B$ which is of type   Desch-Schappacher. Such a perturbation operator $B$  appears for instance in case of control actions exercised through the boundary or at a point of the geometrical domain of  partial differential equations, and also in  many other situations  of internal control (see  \cite{Bou21,grei,had15} and the references therein). Due to the unbounded aspect of the  operator $B,$ the  solution of (\ref{S1}) does not exist, in general, with values in  $X.$ Thus, to confront  this difficulty  the concept of admissibility is needed, which requires the introduction of  interpolating and extrapolating  spaces of the state space $X$.

Our goal in this paper is to investigate the uniform exponential stability of the system  (\ref{S1}). This consists on looking for a set of parameters $\rho$ for which there exists a global $X-$valued  mild solution $x(t)$ of   (\ref{S1}) and is such that $ \|x(t)\| \le K e^{-\sigma t}\|x_0\|,\; \forall t\ge0 $ for some constants $K, \sigma>0$. As an application, one can consider the stabilization of bilinear systems by means of  switching controllers, which leads to a closed-loop system like (\ref{S1}). This problem has been considered in \cite{Liu} for a bounded operator $B$. The case of a Miyadera-Voigt type operator has been investigated in \cite{Ouz17}. Moreover, in \cite{Ammari}   the case of $1-$admissibility in Banach space has been considered. However, the $1-$admissibility assumption prevents us to consider the case of Hilbert state  space as in this case the operator $B$ will be necessary bounded (see \cite{G}). In  other words, the $1-$admissibility condition excludes several applications that are also available in Hilbert space. Moreover, in \cite{Ammari} it was assumed that  $D\big ( (A_{-1}-\rho B)|X \big ) =D(A_{-1})\cap D(B|X),$ which played an essential role in the proofs of the stabilization results (in a technical point of view). Unfortunately, there are  several examples in which  this domain  condition is not fulfilled (see e.g Examples $1 \& 2$). In this paper, we will   rather use the $p-$admissibility property with $p\ge 1$. Then we  introduce  new sufficient conditions for uniform exponential stability of system (\ref{S1}),  which are easily checkable. In the sequel, we proceed as follows: The main results of this paper are contained  in Section $2$. In Section $3$, we provide  applications to feedback stabilization of  bilinear  heat  and transport equations.

 \section{Exponential stability}

In this section, we state and prove our two main stabilization results. We start by introducing the necessary tools  regarding the notion of admissibility in connection with the generation results and then provide some a priori estimations of the solution.
 \subsection{Preliminary results}
As pointed out in the introduction,  the unbounded aspect of the  operator $B$ do not guarantee the existence of an $X-$valued  solution $x(t)$ of (\ref{S1}). However, one may extend the system at hand  in a larger (extrapolating)  space $X_{-1}$ of the state space $X$ in which  the existence of the solution $x(t)$  is ensured and then give the required  admissibility conditions of $B,$ so that the solution $x(t)$ lies in $X$. Classically, the spaces $X_1$ and $X_{-1}$ are  defined as follows: $X_1:=(D(A),\|\cdot\|_1),$ where $\|x\|_1:=\|(\lambda I-A)x\|_X, \, x\in D(A)$ for some $\lambda$  in the resolvent set $\rho(A)$ of $A$, and $X_{-1}$ is the completion of $X$ with respect to the norm $\|x\|_{-1}:=\|(\lambda I-A)^{-1}x\|_X, \, x\in X.$  These spaces are independent of the choice of $\lambda$ and are related by the following continuous and dense embedding: $X_1\hookrightarrow X\hookrightarrow X_{-1}.$ That way the unbounded operator $B$ becomes  bounded   from $X$ to the extrapolating   space $X_{-1},$ i.e, $B\in\mathcal{L}(X,X_{-1}).$ Thus, in order to give a meaning to  solutions of (\ref{S1}), we have to use the fact  that the semigroup $S(t)$ can be extended to a  $C_0-$semigroup $S_{-1}(t)$ on $X_{-1},$  whose  generator $A_{-1}$ has  $D(A_{-1})=X$ as  domain and is such that $A_{-1}x=Ax, $ for any $x\in D(A).$
 By definition, $(X_{-1},\|.\|_{-1})$ is a Banach space and one  can easily show that the operator $A$ is bounded from $(D(A), \|.\|_X)$ into the Banach space $(X_{-1},\|.\|_{-1})$, so that it can be extended to an operator $A_{-1}\in\mathcal{L}(X,X_{-1}),$ generating a $C_0-$semigroup  $S_{-1}(t)$ on $X_{-1}$ and which is given by $S_{-1}(t)=(\lambda-A_{-1})S(t)(\lambda-A_{-1})^{-1}.$ Indeed, it clear that  $S_{-1}(t)$  is a $C_0-$semigroup  on $X_{-1}.$ Moreover, to show  that its generator is $A_{-1}$, we show that for all $x\in X,$ $\lim_{t\to0}\frac{S_{-1}(t)x-x}{t}=A_{-1}x.$ We first show this for $x\in D(A)$. This case follows from the fact that $ \lim_{t\to0}\frac{S_{-1}(t)x-x}{t}=\lim_{t\to0}\frac{S(t)x-x}{t}=Ax.$ Now for $x\in X$, we use the fact that $R(\lambda,A)x\in D(A)$ to deduce that $R(\lambda,A_{-1})\lim_{t\to 0}\frac{S_{-1}(t)x-x}{t}=R(\lambda,A_{-1})A_{-1}x,$ which gives the desired result.  Moreover, if the semigroup $S(t)$ is a contraction, then so is $S_{-1}(t)$. Indeed, having in mind  that  for all $ \lambda \in \rho(A_{-1}) $ and for all $x\in X_{-1},$  we have $(\lambda I-A_{-1})^{-1} x\in X,$ and  $(\lambda I-A_{-1})^{-1} x = (\lambda I-A)^{-1} x,\; \forall x\in X,$ we can see that for all  $x\in X$ we have
$$
\|S_{-1}(t)x\|_{-1}=\|(\lambda I-A)^{-1} S(t)x\|_X\le \|(\lambda I-A)^{-1} x\|_X=\|x\|_{-1},
$$
thus, by density of $X$ in $X_{-1}$, we conclude that the semigroup $S_{-1}(t)$ is a contraction on $X_{-1}$.\\
Recall that for any given initial state $x_0\in X$,  a mild solution of (\ref{S1}) is  an $X-$valued continuous function $x$ on $[0,T]$ satisfying the following variation of parameters formula:
\begin{equation*}\label{VPF}
 x(t)=S(t)x_0-\rho \int_{0}^{t}S_{-1}(t-s)Bx(s) ds,\;\;\forall t\geq0,
\end{equation*}
which  always makes sense in $X_{-1}$.
 The system (\ref{S1}) can be rewritten in the large space  $X_{-1}$ in the  following  abstract form:
\begin{eqnarray}\label{S2}
	\begin{cases}
		\dot{x}(t)=A_{-1}x(t)-\rho Bx(t),\\
		x(0)=x_0.\\
	\end{cases}
\end{eqnarray}
which is well-posed in $X$ whenever  $A-\rho B$  is the generator of a $C_0-$semigroup on $X$ (cf. \cite{eng}, Section II.6). The well-posedeness of systems like  (\ref{S1})  has been studied in many works using different approaches (see e.g. \cite{Adler,Berra09,Desch,Idri03,Mragh14,G}).

The next result provides sufficient conditions on a  Desch-Schappacher perturbation $B$ to guarantee the existence and uniqueness of the mild solution of  (\ref{S1}) (see \cite{Adler} $\&$ (\cite{eng}, p. 183)).

\begin{theorem}\label{ThD}
  Let $A$ be the generator of a $C_0-$semigroup $S(t)$ on $X$ and let $B\in\mathcal{L}(X,X_{-1})$ be $p-$admissible  for some  $1\leq p<\infty$, i.e., there is a $T> 0$ such that
  \begin{equation}\label{ad}
  \displaystyle\int_{0}^{T}S_{-1}(T-t)Bu(s)ds\in X,~~~~\forall u\in L^p(0,T;X).
  \end{equation}
Then for any $\rho$,  the operator $(A_{-1}-\rho B)|_X$  defined on the domain $ D((A_{-1}-\rho B)|_X):=\{  x\in X:(A_{-1}-\rho B)x\in X \}$ by
   \begin{equation}\label{part}
 (A_{-1}-\rho B)|_Xx:=A_{-1}x-\rho Bx,\;\; \forall x\in D((A_{-1}-\rho B)|_X)
 \end{equation}
  is the generator of a $C_0-$semigroup $(T(t))_{t\geq0}$ on $X$, which  verifies the variation of parameters formula
$$
T(t)x=S(t)x-\rho\int_{0}^{t}S_{-1}(t-s)BT(s)xds, \;\;\forall t\geq0, \;\forall x\in D((A_{-1}-\rho B)|_X).
$$
  \end{theorem}

An operator $B\in\mathcal{L}(X,X_{-1})$ satisfying the condition (\ref{ad}) is called a Desch-Schappacher operator or perturbation. Moreover, the operator  defined by (\ref{part}) is the part $(A_{-1}-\rho B)|_X$  of $(A_{-1}-\rho B) $ on $X$ \big(see (\cite{paz}, p. 39) and (\cite{eng}, p. 147)\big).\\

\begin{remark}
Notice that since $W^{1,p}(0,T; X)$ is dense in $ L^p(0,T;X),$ the range condition (\ref{ad}) is equivalent to the existence of some $M> 0$ such that
\begin{equation}\label{adT}
\left\|\displaystyle\int_{0}^{T}S_{-1}(T-s)Bu(s)ds\right\|_X\leq M  \|u\|_{L^p(0,T; U)},
~~\forall u\in W^{1,p}(0,T;X),	
\end{equation}
with $\|u\|_{L^p(0,T; X)}=\displaystyle\left( \int_{0}^{T}\|u(t)\|^p_Xdt\right)^{\frac{1}{p}}$.\\
\end{remark}
\begin{remark}\label{R2}
Note that if the operator $B\in\mathcal{L}(X,X_{-1})$ is $p-$admissible in $[0,T],$ then it is so in 	$[0,t]$ for any $t\in [0,T].$
In other words, if (\ref{adT}) holds then for all $t\in[0,T]$  we have the following inequality
\begin{equation}\label{adt}
	\left\|\int_{0}^{t}S_{-1}(t-r)Bu(r)dr\right\|_X \leq M\|u\|_{L^p(0,t;X)}\cdot	
\end{equation}
	Indeed, if for all $t\in [0,T]$ and $u\in L^p(0,T;X),$ we define $u_t\in L^p(0,T;X)$ by
		$$
		u_t(r):=\left\{
		\begin{array}{ll}
			0, & \mbox{for} \; r\in [0,T-t] \\
			\\
			u(r+t-T),& \mbox{for}\;   r\in(T-t,T]
		\end{array}
		\right.
		$$
		then, observing that $\int_{0}^{t}S_{-1}(t-r)Bu(r)dr=\int_{0}^{T}S_{-1}(T-r)Bu_t(r)dr\in X,$  it comes from the $p-$admissibility of $B$ that
		$	\left\|\int_{0}^{t}S_{-1}(t-r)Bu(r)dr\right\|_X\le M\|u_t\|_{L^p(0,T;X)}\cdot
			$		Thus for all $t\in[0,T]$, we have $	
			\|\int_{0}^{t}S_{-1}(t-r)Bu(r)dr \|_X \leq M\|u\|_{L^p(0,t;X)}\cdot	
		$

		\end{remark}
	
Let us now show the  following lemma that will be needed in the sequel.

\begin{lemma}\label{lem1}
Let assumptions of Theorem \ref{ThD} hold. Then for any  $0<\rho<\frac{1}{T^{\frac{1}{p}}M}$, the mild solution $x(t)$ of the system (\ref{S1}) satisfies  the following estimate
\begin{equation}\label{x(t)}
	\|x(.)\|_{L^p(0,T;X)} \le\frac{T^{\frac{1}{p}}}{1-\rho T^{\frac{1}{p}} M}\|x_0\|_X,\; \forall x_0\in X	
\end{equation}
and
\begin{equation*}\label{tad}
	\left\|	\displaystyle\int_{0}^{t}S_{-1}(t-s)Bx(s)ds\right\|_X\leq M_{\rho}\|x_0\|_X,\;\;\forall t\in[T,2T], \; \forall x_0\in X,
\end{equation*}
with $M_\rho:=\frac{MT^{\frac{1}{p}}}{1-\rho T^{\frac{1}{p}} M}  \left(2+\rho MT^{\frac{1}{p}}\right).$
	\end{lemma}

\begin{proof}
Let $x_0\in D((A_{-1}-\rho B)|_X).$  From Theorem \ref{ThD}, we know that  the system (\ref{S1}) admits a unique  mild solution $x(t)$ which is given by
 \begin{equation}\label{form}	
 x(t)=S(t)x_0-\rho\displaystyle\int_{0}^{t}S_{-1}(t-s)Bx(s)ds,~~  \forall t\geq 0.
 \end{equation}
Let us  estimate $\|x(\cdot)\|_{L^p(0,T;X)}$. From (\ref{form}), we get via Minkowski's inequality
$$ \|x(.)\|_{L^p(0,T;X)}\leq \left( \displaystyle\int_{0}^{T}\|S(t)x_0\|^p_Xdt\right) ^{\frac{1}{p}}+\rho\left( \int_{0}^{T}\left\|\int_{0}^{t}S_{-1}(t-s)Bx(s)ds\right\|^p_X\right)^{\frac{1}{p}}dt$$
Then from Remark \ref{R2} we derive
\begin{eqnarray*}
	\|x(\cdot)\|_{L^p(0,T;X)} &\le&T^{\frac{1}{p}}\|x_0\|_X+\rho T^{\frac{1}{p}}M\|x(.)\|_p,	
\end{eqnarray*}
where
$\|x(\cdot)\|_{L^p(0,T;X)}:=\displaystyle\left( \int_{0}^{T} \|x(\tau)\|^p_X d\tau\right)^{\frac{1}{p}},$
which gives the estimate (\ref{x(t)})
for any $0<\rho<\frac{1}{T^{\frac{1}{p}} M}.$\\	
 Now, since the mapping  $x_0 \mapsto x(t)$ defines a $C_0-$semigroup $T(t)$ on $X$,  the mapping $x_0\mapsto x(\cdot)=T(\cdot) x_0 $ is continuous from $X$ to $L^p(0,T;X).$  Then the estimate (\ref{x(t)}) holds by density for any $x_0\in X.$
Let $x_0\in X$, and let us  write for any  $t\in[T,2T]$,
 	\begin{eqnarray*}
		\displaystyle\int_{0}^{t}S_{-1}(t-s)Bx(s)ds&=&\int_{0}^{T}S_{-1}(t-s)Bx(s)ds+\int_{T}^{t}S_{-1}(t-s)Bx(s)ds\\&:=& L_1+L_2
	\end{eqnarray*}
	Then we consider the two terms of the sum separately.  For the first one,  the admissibility of $B$ together with the contraction property of $S_{-1}(t)$ yields
	\begin{eqnarray}\label{L_1}
		\|L_1\|_X&=&\left\|S_{-1}(t-T)\int_{0}^{T}S_{-1}(T-s)Bx(s)ds\right\|_X\leq M\|x(\cdot)\|_{L^p(0,T;X)}
	\end{eqnarray}
	For the second term, observing that $L_2=\int_{0}^{t-T}S_{-1}(t-T-\tau)Bx(\tau+T)d\tau$,
	we obtain again from the admissibility of $B$
	$$\|L_2\|_X\leq M\|x(.+T)\|_{L^p(0,T;X)}\cdot$$
	 Based on the V.C.F (\ref{form}), it follows directly from Lemma (\ref{x(t)}) that for all $t\ge0,$ we have
	\begin{equation}\label{x(t)_X}
		\|x(t)\|_X\leq\left(1+ \frac{\rho M T^{\frac{1}{p}}}{1-T^{\frac{1}{p}}\rho M} \right) \|x_0\|_X, \forall x_0\in X
	\end{equation}
	for any  $0<\rho<\frac{1}{T^{\frac{1}{p}} M}.$
	Using the last estimate, we derive  the following  inequalities:
	\begin{equation}\label{L_2}
		\|x(.+T)\|_{L^P(0,T;X)}
		\le  T^{\frac{1}{p}}  \left(1+\frac{\rho M^2T^{\frac{1}{p}}}{1-\rho MT^{\frac{1}{p}}} \right)\|x_0\|_X.
	\end{equation}
	Combining (\ref{L_1}) and (\ref{L_2}),  we obtain the desired estimate:
	\begin{equation*}\label{L}
		\left\|	\displaystyle\int_{0}^{t}S_{-1}(t-s)Bx(s)ds\right\|_X\leq M_\rho \|x_0\|_X\cdot
	\end{equation*}
	 \end{proof}
\subsection{A direct approach}
     \begin{theorem}\label{T2}
      Let $B\in {\cal L}(X,X_{-1})$ and let $A$ be the infinitesimal generator of a linear $C_0-$semigroup of contractions $S(t)$ on $X,$ and assume that for some $T>0,$ we have
      \begin{enumerate}
   \item[(i)] there exists $1<p<\infty$   such that for all $u\in L^p(0,T; X),$ we have\\ $ \displaystyle\int_{0}^{T}S_{-1}(T-s)Bu(s)ds\in X,$
                \item[(ii)]  for some
 $\delta>0$ we have
           \begin{equation}\label{obs}
            \displaystyle\int_{0}^{T}Re\left\langle S(t)x,B^*S(t)x\right\rangle_{X}dt\geq \delta\|S(T)x\|^2_{X},\;\; \forall x\in X.
             \end{equation}
            \end{enumerate}
    Then  there is a $\rho_1>0$ such that for all $\rho\in(0,\rho_{1})$, the system (\ref{S1}) is  exponentially stable on $X$.
    \end{theorem}

   \begin{proof}
Foe any $\rho>0,$ we set $A_{\rho B}:=(A_{-1} -\rho B)|_X$. According to assumption $(i),$  we deduce from Theorem $1$ that
the system (\ref{S1}) admits a unique mild solution which is given,  for $x_0\in D(A_{\rho B}), $  by the variation of parameters formula (see \cite{Desch}):
     \begin{equation}\label{f}
       x(t)=S(t)x_0-\rho\displaystyle\int_{0}^{t}S_{-1}(t-s)Bx(s)ds,~~  \forall t\geq 0.
      \end{equation}
  For $\lambda\in \rho(A)$ $(\rho(A)$  is the resolvent set of $A),$ we consider the system (\ref{S1}) with  $B_{\lambda}:=\lambda R(\lambda,A_{-1})B$ instead of $B.$  Observing that the operator $B_{\lambda}$ is bounded, we deduce that   the corresponding system admits a unique mild solution denoted  by $x_{\lambda},$ which  satisfies the following formula
  \begin{equation}\label{VPF}
  	x_{\lambda}(t)=S(t)x_0-\rho\displaystyle\int_{0}^{t}S_{-1}(t-s)B_{\lambda}x_{\lambda}(s)ds,~~  \forall t\geq 0.		
  \end{equation}
 We claim that $x_{\lambda}(t)$ converges to $x(t)$ as $\lambda\to +\infty.$ Indeed, for all $t>0$, we have
  \begin{eqnarray*}
  	x_{\lambda} (t)-x(t)&=&\rho \int_{0}^{t}S_{-1}(t-s)B_{\lambda}x_{\lambda}(s)-\rho \int_{0}^{t}S_{-1}Bx(s)ds\\&=&\rho \int_{0}^{t}S_{-1}(t-s)B_{\lambda}(x_{\lambda}(s)-x(s))ds+\rho \int_{0}^{t}S_{-1}(t-s)B_{\lambda}x(s)ds-\rho \int_{0}^{t}S_{-1}(t-s)Bx(s)ds.	 	
  \end{eqnarray*}
 Then, using (\ref{adT}), this yields for all $t\in[0,T]$
  \begin{eqnarray*}
  	\|x_{\lambda} (t)-x(t)\|_X&\leq&\rho M\|x_{\lambda}(.)-x(.)\|_{L^p(0,t;X)}+\rho\left\| \int_{0}^{t}S_{-1}(t-s)B_{\lambda}x(s)ds- \int_{0}^{t}S_{-1}(t-s)Bx(s)ds\right\|_X,	
  \end{eqnarray*}
  which by integrating gives for all $t\in [0,T]$
  	$$
  \|x_{\lambda} (t)-x(t)\|^p_{L^p(0,T;X)}\leq T(2\rho M)^p\|x_{\lambda}(.)-x(.)\|^p_{L^p(0,T;X)}+
  $$
  $$(2\rho)^p\left\| \int_{0}^{t}S_{-1}(t-s)B_{\lambda}x(s)ds-
     \int_{0}^{t}S_{-1}(t-s)Bx(s)ds\right\|^p_X.	
  $$
  It follows that
  \begin{eqnarray*}
  	\|x_{\lambda} (.)-x(.)\|^p_{L^p(0,T;X)}\leq \rho C_{\rho}\int_{0}^{T}\left\| \int_{0}^{t}S_{-1}(t-s)B_{\lambda}x(s)ds- \int_{0}^{t}S_{-1}(t-s)Bx(s)ds\right\|^p_Xdt,	
  \end{eqnarray*}
  with $C_{\rho}:=\frac{(2\rho)^p}{1-T(2\rho M)^p}.$\\
  It is clear that $\lim_{\lambda\to\infty}  \int_{0}^{t}S_{-1}(t-s)B_{\lambda}x(s)ds= \int_{0}^{t}S_{-1}(t-s)Bx(s)ds$ in $X$ and we have
  $$
  	\left\| \int_{0}^{t}S_{-1}(t-s)B_{\lambda}x(s)ds- \int_{0}^{t}S_{-1}(t-s)Bx(s)ds\right\|_X \leq \left\| \int_{0}^{t}S_{-1}(t-s)B_{\lambda}x(s)ds\right\|_	X+
  $$
  $$\left\|\int_{0}^{t}S_{-1}(t-s)Bx(s)ds\right\|_X.
  $$		
 Moreover,  by the admissibility assumption we get
  $$\left\| \int_{0}^{t}S_{-1}(t-s)B_{\lambda}x(s)ds- \int_{0}^{t}S_{-1}(t-s)Bx(s)ds\right\|_X\leq 2M\|x(.)\|_{L^p(0,t;X)}.$$ Then, according to the dominated convergence theorem we have
  \begin{eqnarray*}
  	\lim_{\lambda\to\infty}	\|x_{\lambda} (.)-x(.)\|^p_{L^p(0,T;X)}&\leq&\lim_{\lambda\to\infty}\int_{0}^{T}\left\|\rho \int_{0}^{t}S_{-1}(t-s)B_{\lambda}x(s)ds-\rho \int_{0}^{t}S_{-1}(t-s)Bx(s)ds\right\|_Xdt\\&=&\int_{0}^{T}\lim_{\lambda\to\infty}\left\|\rho \int_{0}^{t}S_{-1}(t-s)B_{\lambda}x(s)ds-\rho \int_{0}^{t}S_{-1}(t-s)Bx(s)ds\right\|_Xdt\\&=&0.
  \end{eqnarray*}
Let  $x_0\in D(A_{\rho B})$ be fixed. Thus for all $t>0 $ we have
         \begin{equation}\label{dx}
        \frac{d}{dt}\Vert x_{\lambda}(t)\Vert_X^2\leq -2\rho \mathcal{R}e \langle B_{\lambda}x_{\lambda}(t),x_{\lambda}(t)\rangle_X,\;\; \forall t>0\cdot
        \end{equation}
          For all $t> 0,$ we have  the following equality
\begin{eqnarray*}
	\left\langle B_{\lambda}S(t)x_0,S(t)x_0\right\rangle_X &=&\left\langle B_{\lambda}S(t)x_0,S(t)x_0-x_{\lambda}(t)\right\rangle_{X}\\&+&\left\langle B_{\lambda}S(t)x_0-B_{\lambda}x_{\lambda}(t),x_{\lambda}(t)\right\rangle_{X}+ \left\langle B_{\lambda}x_{\lambda}(t),x_{\lambda}(t)\right\rangle_{X}
\end{eqnarray*}
which gives
\begin{eqnarray*}
	\left\langle S(t)x_0,B^*\lambda R^*(\lambda,A)S(t)x_0\right\rangle_X &=&\left\langle S(t)x_0,B^*\lambda R^*(\lambda,A)\left( S(t)x_0-x_{\lambda}(t)\right) \right\rangle_{X}\\&+&\left\langle S(t)x_0-x_{\lambda}(t),B^*\lambda R^*(\lambda,A)x_{\lambda}(t)\right\rangle_{X}+ \left\langle B_{\lambda}x_{\lambda}(t),x_{\lambda}(t)\right\rangle_{X}
\end{eqnarray*}
Let us estimate each term of this last expression.\\
We deduce from $(i)$  that   for some  constant $M>0$ and  for all $ u\in L^p(0,T;X)$,  we have
 \begin{equation}\label{S-1}
  \left\|\displaystyle \int_0^T S_{-1} (T-s) Bu(s)  ds\right\|_X \le M\|u\|_{L^p(0,T; X)}
    \end{equation}
The formula (\ref{VPF}) combined  with the estimate (\ref{adt}), gives
\begin{equation*}
	\|S(t)x_0-x_{\lambda}(t)\|_X  =\rho \left\|\int_0^t S_{-1}(t-s) B x_{\lambda}(s) ds \right\|_X \leq\rho M \|x_{\lambda}(.)\|_{L^p(0,T;X)}, \forall t\in [0,T].
\end{equation*}
Then, according to Lemma \ref{lem1}, we conclude that
\begin{equation}\label{S}
	\|S(t)x_0-x_{\lambda}(t)\|_X \le \frac{\rho M T^{\frac{1}{p}}}{1-\rho T^{\frac{1}{p}} M} \| x_0\|_X	
\end{equation}
 For every $ t>0,$ we have
\begin{eqnarray*}
\mathcal{R}e \left\langle S(t)x_0,B^*\lambda R^*(\lambda,A)S(t)x_0\right\rangle_{X} &\leq& C\|S(t)x_0\|_X\|B^*  \|_{{\cal L} (X_{-1},X) }\|\lambda R^*(\lambda,A)\|_{\mathcal{L}(X)} \|S(t)x_0-x_{\lambda}(t)\|_{X}\\&+& C\|B^*  \|_{{\cal L} (X_{-1},X) }\|\lambda R^*(\lambda,A)\|_{\mathcal{L}(X)}	\|x(t)\|_{X} \| S(t)x_0-x(t)\|_{X}\\&+& \mathcal{R}e \left\langle B_{\lambda}x_{\lambda}(t),x_{\lambda}(t)\right\rangle_{X},
\end{eqnarray*}
where $C$ is a positive constant.\\
Using the fact that $S(t)$ is a contraction, it comes
\begin{eqnarray*}
\mathcal{R}e \left\langle S(t)x_0,B^*\lambda R^*(\lambda,A)S(t)x_0\right\rangle_{X} &\leq& C\|B^*  \|_{{\cal L} (X_{-1},X) }\|x_0\|_X \|S(t)x_0-x_{\lambda}(t)\|_{X}\\&+& C\|B^*  \|_{{\cal L} (X_{-1},X) }\|x(t)\|_{X} \| S(t)x_0-x(t)\|_{X}+ \mathcal{R}e \left\langle B_{\lambda}x_{\lambda}(t),x_{\lambda}(t)\right\rangle_{X}\cdot
\end{eqnarray*}
Using (\ref{x(t)_X}) and (\ref{S}), we deduce that for all $t\in (0,T]$ we have
\begin{eqnarray*}
\mathcal{R}e \left\langle S(t)x_0,B^*\lambda R^*(\lambda,A)S(t)x_0\right\rangle_{X} &\leq& \frac{\rho CL M T^{\frac{1}{p}}}{1-\rho T^{\frac{1}{p}} M} \| x_0\|^2_X \\&+& \frac{\rho CL MT^{\frac{1}{p}}}{1-\rho T^{\frac{1}{p}} M}  \left(1+ \frac{\rho M T^{\frac{1}{p}}}{1-T^{\frac{1}{p}}\rho M} \right) \|x_0\|_X^2+ \mathcal{R}e \left\langle B_{\lambda}x_{\lambda}(t),x_{\lambda}(t)\right\rangle_{X},
\end{eqnarray*}
with $L:= \|B^*\|_{{\cal L} (X_{-1},X)}$.  Then, integrating the last inequality  and using (\ref{dx}) we get for all $t\in [0,t]$

\begin{eqnarray*}
	2\rho\int_{0}^{T}\mathcal{R}e \left\langle S(t)x_0,B^*\lambda R^*(\lambda,A)S(t)x_0\right\rangle_{X} &\leq& \frac{2\rho^2 CL M T^{\frac{1}{p}}}{1-\rho T^{\frac{1}{p}} M} \| x_0\|^2_X \\&+& \frac{2\rho^2 CL MT^{\frac{1}{p}}}{1-\rho T^{\frac{1}{p}} M}  \left(1+ \frac{\rho M T^{\frac{1}{p}}}{1-T^{\frac{1}{p}}\rho M} \right) \|x_0\|_X^2+ \|x_0\|^2_X-\|x_{\lambda}(T)\|_X^2,
\end{eqnarray*}

Thus, letting $\lambda\to +\infty,$ we drive
 \begin{eqnarray*}
 2\mathcal{R}e\int_{0}^{T} Re \left\langle S(t)x_0,B^*S(t)x_0\right\rangle_{X} dt&\leq& 2\rho^2 C_1\|x_0\|^2_{X}+\|x_0\|^2_X-\|x(T)\|_X^2
 \end{eqnarray*}
 with $C_1=\frac{ MCL T^{1+\frac{1}{p}}}{1-\rho T^{\frac{1}{p}} M} 	\left(2+\frac{\rho MT^{\frac{1}{p}}}{1-\rho T^{\frac{1}{p}}M} \right).$\\
 Applying the inequality (\ref{obs}), it follows that
   \begin{equation}\label{13}
   2\rho\delta\|S(T)x_0\|_X^2-2\rho^2 C_1\|x_0\|^2_{X} \leq\|x_0\|^2_X-\|x(T)\|_X^2
  \end{equation}
Using Lemma \ref{lem1},  we deduce via  the  variation of constants formula (\ref{f}) that for all $t\in [T,2T],$ we have
$$ \begin{array}{lll}
	\|x(t)\|_X &\le   \|S(T)x_0\|_X+ \rho \|\int_0^t S_{-1}(t-s) B x(s) ds\|_X& \\
\\
	& \le \|S(T)x_0\|_X+ \rho M_\rho \|x_0\|_X.&
\end{array}
$$
By reiterating the processes for $t\in[kT,(k+1)T], k\geq1, $ we deduce that
$$ 	\|x(t)\|_X \le   \|S(T)x(kT)\|_X+ \rho M_\rho \|x(kT)\|_X.$$
Then for all $k\geq1, $ we have
\begin{equation}\label{zS(T)}
	\|x((k+1)T)\|_X^2
	\le  2\|S(T)x(kT)\|_X^2 + 2 \rho M_\rho  \|x(kT)\|^2_X.
\end{equation}
Moreover, (\ref{13}) becomes  \begin{equation}\label{kT}
	2\rho\delta\|S(T)(kT)\|_X^2-2\rho^2 C_1\|x(kT)\|^2_{X} \leq\|x(kT)\|^2_X-\|x((k+1)T)\|_X^2
\end{equation}
  This together with   (\ref{zS(T)}) implies
$$
\rho \delta\bigg ( \| x((k+1)T)\|_X^2 - 2 \rho M_\rho \|x(kT)\|^2_X \bigg ) - 2C_1\rho^2 \|x(kT)\|^2_X\le
$$
$$
\| x(kT)\|_X^2 -\| x((k+1)T)\|_X^2.
$$
Hence
$$
(1+\rho \delta)   \| x((k+1)T)\|_X^2 \le  \bigg (2\delta \rho^2M_\rho +2C_1\rho^2  + 1 \bigg ) \| x(kT)\|_X^2 , \; k\ge 0\cdot
$$
This implies
   \begin{equation*}\label{11}
     \| x\left( (k+1)T\right) \|^2_{X}\leq C_2\| x(kT)\|^2_{X}
     \end{equation*}
 where $C_2=\frac{2 \rho^2 ( \delta M_\rho+C_1) + 1}{1+\rho \delta},$ which is in $(0,1)$ for $\rho\to 0^+$.\\
Since $\|x(t)\|_X$  decreases, we get for $k=E\left(  \frac{t}{T}\right)$  (where $E(.)$ is the integer part function).\\
          $$\|x(t)\|^2_{X}\leq (C_2)^k\|x_0\|^2_{X},$$
       which gives the following exponential decay
       \begin{equation*}\label{exx}
         \|x(t)\|_{X}\leq K e^{-\sigma t} \|x_0\|,\; \forall t\ge 0\cdot
       \end{equation*}
     where $K=(C_2)^{-\frac{1}{2}}$ and $\sigma=\frac{-\mbox{ln}(C_2)}{2T}$. This estimate extends by density to all $x_0\in X.$
 Hence the uniform exponential stability hold for any $0<\rho <\rho_1$, where $\rho_1$ is such that $0<\rho_1 < \frac{1}{T^{\frac{1}{p}}M}$ and  $\frac{2 \rho^2 ( \delta M_\rho+C_1) + 1}{1+\rho \delta}\in (0,1).$
       \end{proof}
{\bf{Remark 1}}\\
 In the case where $Range(BS(t))\subset X,\;\forall t>0$,    the condition (\ref{obs}) is equivalent to the conventional one (see \cite{Ammari,Ouz17}):
 \begin{equation}\label{oobs}
	\displaystyle\int_{0}^{T}Re\left\langle BS(t)x,S(t)x\right\rangle_{X}dt\geq \delta\|S(T)x\|^2_{X},\;\; \forall x\in X.
	\end{equation}
Moreover, the condition (\ref{oobs}) can be weakened if an appropriate decomposition  of  $Range(B)$ is available. This is  the aim of the next section.
\subsection{A range decomposition method}
   Let  $\mathfrak{X}\oplus\mathfrak{X}_{-1}$ be a direct  sum in $X_{-1}$, where  $\mathfrak{X}=i(X) \; (i$  being the canonical injection of $X$ in $X_{-1}),$ so we can write $\mathfrak{X}=X.$ Then for any $C\in\mathcal{L}(X,X_{-1})$ such that $rg(C)\subset\mathfrak{X}\oplus\mathfrak{X}_{-1},$ we set $_XC=:P_{\mathfrak{X}}C,$ where $P_{\mathfrak{X}}$ is the projection of $\mathfrak{X}$  according to  $\mathfrak{X}\oplus\mathfrak{X}_{-1}$. Now, given a pair of operators $(K,L)\in \mathcal{L}(X,X_{-1})\times\mathcal{L}(X,X_{-1})$, the decomposition $\mathfrak{X}\oplus\mathfrak{X}_{-1}$ is said to be admissible for $(K,L)$ if the three following properties hold:\\
  (a) $rg(K)\subset\mathfrak{X}\oplus\mathfrak{X}_{-1}$  and $rg(L)\subset\mathfrak{X}\oplus\mathfrak{X}_{-1}$,\\
  (b) $_XK$ is dissipative on $D((K+L)|_X):=\left\lbrace x\in X : Kx+Lx\in X\right\rbrace, $\\
   (c)  $_XL\in\mathcal{L}(X).$\\
For our stabilization problem,  we will be interested  with admissible decompositions for the pairs $(A_{-1},-\rho B)$ with $ \rho>0$ small enough.
 Note that if the domain of the operator  $(A_{\rho B})|_X$ is independent of $ \rho>0$ (small enough), which is equivalent to $D((A_{\rho B})|_X)= D(A)\cap D(B|_X),$ then for the sum $\mathfrak{X}\oplus\mathfrak{X}_{-1}$ to be admissible  for the pairs $(A_{-1},-\rho B),\; \rho>0,$ it suffices to be  admissible  for the pair $(A_{-1},B).$\\

We are ready to state our second main result.

    \begin{theorem}\label{T1}
  	  Let $A$ be the infinitesimal generator of a linear $C_0-$semigroup of contractions $S(t)$ on $X$ and let $B\in \mathcal{L}(X,X_{-1})$. Let $\mathfrak{X}\oplus\mathfrak{X}_{-1}$ be an admissible decomposition for the pair $(A_{-1},-\rho B)$ for any $\rho>0$ small enough,  and assume that for some $T>0$, the operator   $B$ is $p-$admissible  for some $1<p<\infty$  and satisfies the estimate:
  \begin{equation}\label{obs2}
            \displaystyle\int_{0}^{T}Re\left\langle _XBS(t)x,S(t)x\right\rangle_{X}dt\geq \delta\|S(T)x\|^2_{X},\;\; \forall x\in X,
             \end{equation}
             for some $T, \delta>0.$\\
               Then  there is a $\rho_1>0$ such that
   	the system (\ref{S1}) is  exponentially stable on $X$ for all $\rho\in(0,\rho_{1})$.
   \end{theorem}

   \begin{proof}
 Let $0<\rho < \frac{1}{T^{\frac{1}{p}}M}$, and let $x(t)$ be the unique mild solution of 	the system (\ref{S1}) given for  $x_0\in D((A_{\rho B})|_X)$  by the  formula (\ref{f}). \\
 The admissibility assumption on $ B $  together with Lemma \ref{lem1} implies the following estimate for $t\in [0,T]:$
    \begin{equation}\label{x(t)-S(t)}
  	\|x(t)-S(t)x_0\|_X \le \frac{\rho M T^{\frac{1}{p}}}{1-\rho T^{\frac{1}{p}} M} \| x_0\|_X	
   \end{equation}
 Moreover, observing that $A_{\rho B}x(t)=_X\!\!(A_{\rho B})x(t)$,  we can write
   	\begin{eqnarray*}\label{dxx}
   		\frac{d}{dt}\Vert x(t)\Vert_X^2 = 2\mathcal{R}e\langle_X(A_{-1})x(t)-\rho \;_XBx(t),x(t)\rangle_X,\; \forall t>0\cdot	
   	\end{eqnarray*}
   	Integrating this last equality and using the dissipativeness of $_X(A_{-1})$ gives
   	\begin{equation}\label{inner}
   		2\rho \int_s^t \mathcal{R}e  \langle _XBx(\tau),  x(\tau) \rangle_X d\tau \le \|x(s)\|_X^2-\|x(t)\|_X^2,\; t\ge s\ge 0.
   	\end{equation}
   We have the following equality
   	\begin{eqnarray*}
   		\left\langle _XBS(t)x_0,S(t)x_0\right\rangle_X &=&\left\langle _XBS(t)x_0-\;      _XBx(t),S(t)x_0\right\rangle_{X}\\&+&\left\langle _XBx(t),S(t)x_0-x(t)\right\rangle_{X}+ \left\langle _XBx(t),x(t)\right\rangle_{X}\cdot
   	\end{eqnarray*}
 Then using the fact that the operator $_XB$ is bounded, it comes
   	\begin{eqnarray*}
   		\mathcal{R}e \left\langle _XBS(t)x_0,S(t)x_0\right\rangle_{X} &\leq&  \|_XB\|_{\mathcal{L}(X)}\|x_0\|_{X}\|S(t)x_0-x(t)\|_{X}\\&+&  \|_XB\|_{\mathcal{L}(X)}	 \|x(t)\|_{X}\|S(t)x_0-x(t)\|_{X}+\mathcal{R}e \left\langle _XBx(t),x(t)\right\rangle_{X}
   	\end{eqnarray*}
	The estimate (\ref{x(t)-S(t)}) combined with  (\ref{x(t)_X}), implies 	
   	\begin{eqnarray*}
   	\mathcal{R}e \left\langle _XBS(t)x_0,S(t)x_0\right\rangle_{X} &\leq&  \|_XB\|_{\mathcal{L}(X)}\frac{\rho M T^{\frac{1}{p}}}{1-T^{\frac{1}{p}}\rho M}\|x_0\|_{X}^2\\&+&  \|_XB\|_{\mathcal{L}(X)}	\frac{\rho M T^{\frac{1}{p}}}{1-T^{\frac{1}{p}}\rho M}\left( 1+ \frac{\rho M T^{\frac{1}{p}}}{1-T^{\frac{1}{p}}\rho M} \right)\|x_0\|_X^2\\&+&\mathcal{R}e \left\langle _XBx(t),x(t)\right\rangle_{X},\; \forall t\in [0,T]\cdot
   	\end{eqnarray*}
   Integrating this inequality and	using the inequality (\ref{obs2}),  we deduce that
      $$	\delta\|S(T)x(kT)\|_X^2-\rho C_1\|x(kT)\|^2_{X} \leq\displaystyle\int_{kT}^{(k+1)T}\mathcal{R}e \left\langle _XBx(s),x(s)\right\rangle_{X} ds
      $$
      with $C_1= \frac{ M T^{1+\frac{1}{p}}}{1-\rho T^{\frac{1}{p}} M} 	 \|_XB\|_{\mathcal{L}(X)} \left(2+\frac{\rho MT^{\frac{1}{p}}}{1-\rho T^{\frac{1}{p}}M} \right).$\\
By using Lemma \ref{lem1}, we derive
 $$
   	\rho \delta\bigg ( \| x((k+1)T)\|_X^2 - 2 \rho M_\rho \|x(kT)\|^2_X \bigg ) - 2C_1\rho^2 \|x(kT)\|^2_X\le
   	$$
   	$$
   	\| x(kT)\|_X^2 -\| x((k+1)T)\|_X^2,
   	$$
   	or equivalently
   $$		
   \| x\left( (k+1)T\right) \|^2_{X}\leq C_2\| x(kT)\|^2_{X},
   $$
   	where $C_2=\frac{2 \rho^2 ( \delta M_\rho+C_1) + 1}{1+\rho \delta},$ which lies in $(0,1)$ for $\rho\to 0^+$.\\
 Hence, using  the decreasing   	of  $\|x(t)\|_X$,  we deduce  the following exponential decay
   $$		\|x(t)\|_{X}\leq K e^{-\sigma t} \|x_0\|,\; \forall t\ge 0,
   $$
   	where $K=(C_2)^{-\frac{1}{2}}$ and $\sigma=\frac{-\mbox{ln}(C_2)}{2T}$. This estimate extends by density to all $x_0\in X.$ Thus, taking $\rho_1>0$ such that $0<\rho_1 < \frac{1}{T^{\frac{1}{p}}M}$ and  $C_2\in (0,1)$, we get the result of the theorem.
   \end{proof}
   \section{Examples}
   \begin{example}
 Let $\Omega $ be an open and bounded subset of $\mathbf{R}^d,\; d\ge 1$, and let us consider the  following bilinear equation of diffusion type
   	\begin{eqnarray}\label{ES1}
   		\begin{cases}
   			\frac{\partial }{\partial t}  x = \Delta x+gx + \nu (t)
   			(-\Delta)^{\frac{1}{2}}x~~ &\mbox{in}~~\Omega\times(0,\infty),\\
   			x(t)=0~~&\mbox{on}~~\partial \Omega\times (0,\infty),\\
   			x(0)=x_0~~&\mbox{in}~~\Omega.
   		\end{cases}
   	\end{eqnarray}
where $g\in L^\infty(\Omega),\; \nu$ is a real valued bilinear control and $x(t)=x(\zeta,t)\in L^2(\Omega)$ is the state.
 The system (\ref{ES1})  is an example of fractional equation of diffusion equations type, and may describe  transport processes in complex systems which are slower than the Brownian diffusion. As practical situations  displaying such
anomalous behaviour, let us mention the charge carrier transport in amorphous
semiconductors, the nuclear magnetic resonance diffusometry in percolative and
porous media etc (see \cite{Berra09,Kilbas,Mragh14,Metzler}). Here, we aim to show the exponential stabilization of (\ref{ES1}). Let us observe that  system (\ref{ES1}) can be written in   the form of (\ref{S1}) if we close it by  the switching feedback control $\nu(t)=-\rho {\bf 1}_{\{t\ge 0\:/\: x(t)\ne0\}}$. This is because we have ${\bf 1}_{\{t\ge 0\:/\: x(t)\ne0\}}  (-\Delta)^{\frac{1}{2}}x(t)= (-\Delta)^{\frac{1}{2}}x(t),\; \forall t\ge 0.$. Let us  take the state space $X=L^2(\Omega)$ (endowed with its  natural scalar product $\langle\cdot,\cdot\rangle_X$), and consider the control operator $B= (-\Delta)^{\frac{1}{2}}$ and the system's operator $A=\Delta+ g I $  with $D(A)=H^2(\Omega)\cap H^1_0(\Omega)$. The  operator $A$ generates an analytic semigroup $S(t)$ on $X$ (see \cite{eng}, p. 107 and p. 176) which is given by the following variation of constants formula:
	$$S(t)x=S_0(t)x+\int_0^t S_0(t-s)g(\xi) S(s)x ds,\;\; t\ge 0,$$
where $S_0(t)$ is the semigroup generated by $A$ with  $g=0$.\\
Let us verify the assumptions of Theorem $2$. In order to make the computation easier,  we restrict our self to the mono-dimension case, thus we consider $\Omega =(0,1)$. In this case the semigroup $(S_0(t))$ is is given by
$$
S_0(t)x=\sum_{j\geq1} e^{-\alpha_jt}  \langle x,\phi_j \rangle_X \; \phi_j, \;  \forall x\in L^2(\Omega)$$
   	with $\alpha_j=j^2\pi^2, j\geq1$ is the set of eigenvalues of $-\Delta$ with  the corresponding orthonormal basis  of $L^2(\Omega)$:  $\phi_j(x)=\sqrt{2}\sin(j\pi x)$.  Moreover, the semigroup $S(t)$ is a contraction if in addition
   $$\int_\Omega g(\xi) y^2(\xi) d\xi \le \|y\|_{H_0^1(\Omega)}^2,\; \forall y\in H_0^1(\Omega).$$
   Thus, in the sequel we suppose this condition satisfied.    Then the   operator $B$ can be expressed as $$Bx=\sum_{j\geq1}\alpha_j^{\frac{1}{2}}\langle x,\phi_j \rangle_X \; \phi_j,~x\in L^2(\Omega). $$
   	Here,  $B$ is unbounded on $L^2(\Omega)$  and it is bounded from  $L^2(\Omega)$ onto the space $X_{-1}$ defined as the completion of $L^2(\Omega)$ for the norm $\|y\|=\big (\displaystyle \sum_{j\ge 1} \frac{1}{\alpha_j} \langle y,\phi_j\rangle^2 \big )^{\frac{1}{2}}, \; \forall y\in L^2(\Omega)$,   	which can be also interpreted as the dual space of $D((-\Delta)^{\frac{1}{2}})$ with respect to the $L^2(\Omega)-$topology (the  space $L^2(\Omega)$ being the pivot space).      Note also that the space $D((-\Delta)^{\frac{1}{2}})$ can be doted with the norm  $\|x\|_{D((-\Delta)^{\frac{1}{2}})}=\bigg ( \displaystyle \sum_{j\ge 1} \alpha_j|\langle x,\phi_j\rangle_X|^2 \bigg )^{\frac{1}{2}}$.\\
     Let $p>1,\, T>0$ and let $u\in L^p(0,T;X).$ It follows from the fact that 
     $(-\Delta)^{\frac{1}{2}}\in\mathcal{L}(X,X_{-1}), $  that  the $X_{-1}-$ valued integral $\int_{0}^{T}S_{-1}(T-s)(-\Delta)^{\frac{1}{2}}u(s)ds$ is well-defined (see \cite{eng}, Theorem 5.34). Moreover, since the semigroup  $(S(t))_{t\geq0}$ is  analytic, then so is $((S_{-1}(t))_{t\geq0}$. This implies that  $S_{-1}(\frac{T-s}{2})(-\Delta)^{\frac{1}{2}}u(s)\in X, \; \forall s\in [0,T)$ (see \cite{eng}, p. 101).      Then  we have 	
     $\displaystyle\int_{0}^{T}S_{-1}(T-s)(-\Delta)^{\frac{1}{2}}u(s)ds\in X,$ which gives the $p-$admissibility of $B$ (see \cite{Nagel}, Prop. 3.3 and \cite{van}, Lemma. 4.3.9).\\
     For all $x\in L^2(\Omega),\; t\ge 0 $ and $j\ge 1$, we have
 \begin{eqnarray*}
|\langle  \int_{0}^t S_0(t-s)g S(s) x ds,\phi_j \rangle_X| &=&
 \int_0^t \langle S_0(t-s)g S(s) x ,\phi_j \rangle ds=
 |\int_0^t \langle g S(s) x, e^{-\alpha_j (t-s)}\phi_j \rangle_X ds|\\
&\le& \|g\|_{L^\infty (\Omega)} \| x\|_X \; \frac{1-e^{-\alpha_j t}}{\alpha_j}.
\end{eqnarray*}
We deduce that
\begin{equation}\label{Sg}
|\langle S(t)x,\phi_j \rangle_X|\le e^{-\alpha_j t} \| x\|_X + \|g\|_{L^\infty (\Omega)}  \; \frac{1-e^{-\alpha_j t}}{\alpha_j} \; \| x\|_X,\; t\ge 0,\; j\ge 1.
\end{equation}
Now for any $x\in X$, we have
$$BS(t)x=\sum_{j\geq1}  \alpha_j^{\frac{1}{2}} \left\langle S(t)x ,\phi_j\right\rangle_X \phi_j.$$
This combined with (\ref{Sg}) implies that $BS(t)x\in X$ for all $x\in X$ and for any $t>0$. \\
Now, using the series expansion of  $BS(t)x $ for  $x\in X$, we get
\begin{eqnarray*}
\left\langle BS(t)x,S(t)x\right\rangle_X  &=& \sum_{j\geq1}\alpha_j^{\frac{1}{2}}   \left\langle S(t)x,\phi_j\right\rangle_X^2  \\&\geq& \|S(t)x\|_X^2 \\&\geq& \|S(T)x\|^2_X,\; \forall t\in[0,T].
\end{eqnarray*}
It follows that the assumption (\ref{obs}) is fulfilled.\\
We conclude by Theorem \ref{T2} that for $\rho>0$ small enough,  the control $\nu(t)=-\rho {\bf 1}_{\{t\ge 0,\:\: x(t)\ne0\}}$ guarantees the uniform exponential stability of the system (\ref{ES1}).
\end{example}
\begin{example}
Consider the following system
	\begin{eqnarray*}
		(S_0)	\begin{cases}
			\frac{\partial }{\partial t} (\zeta,t)= \frac{\partial }{\partial \zeta} x(\zeta,t) - \alpha x(\zeta,t) + \nu(t) h(\zeta)x(\zeta,t)\;\;& \mbox{in}\; (0,1)\times (0,\infty),\\
			x(1,t)=0\;\;& \mbox{in}\; (0,\infty),\\
x(\cdot,0)=x_0\in L^2(0,1)
		\end{cases}
	\end{eqnarray*}
	where $X=L^2(0,1), \alpha >0$ and   $ h\in L^{\infty}(0,1)$ is such that $h\geq c>0,$ for some positive constant $c$. Here we can take $A=\frac{d}{d\zeta}-\alpha \, i\!d$ with domain $D(A):=\left\lbrace x\in H^1(0,1) : x(1)=0 \right\rbrace.$\\
The operator $A$ is the generator of  a contraction semigroup $(S(t)_{t\geq 0})$ given by
$$
\big ( S(t) x \big) (\zeta)=\left\{
\begin{array}{ll}
	e^{-\alpha t } x(\zeta+t) & \mbox{if} \; \zeta+t\le 1, \\
	\\
	0 & \mbox{else.}
\end{array}
\right.
$$
	According to previous theorems, the system $(S_0)$ is exponentially stablilizable  by  the switching feedback control $\nu(t)=-\rho {\bf 1}_{\{t\ge 0\:/\: x(t)\ne0\}}$. Indeed, here the semigroup $S(t)$ is a contraction (so that $\|S(t)\|$ is decreasing) and the linear bounded operator $B_1:= h\: i\!d$  is a bounded linear operator $(h\in L^{\infty})$ and satisfies the observation condition (since  $h\ge c>0)$. Let us now consider the following system
	\begin{eqnarray*}
		(S_1)	\begin{cases}
			\dot{x}(t)=x_{\zeta}(t)- \alpha x(\zeta,t)+\nu(t) h(\zeta)x(t)\;\;& \mbox{in}\; (0,1)\times (0,\infty)\\
			x(1,t)+\epsilon \psi( x(t)) =0\;\;& \mbox{in}\; (0,\infty)
		\end{cases}
	\end{eqnarray*}	
 where $\psi: X \rightarrow \mathbf{R}$ is a non null linear functional of $X$. This  may be seen as a perturbed version of $(S_0)$ on its boundary conditions.  According to Riesz representation, one can assume that $\psi(x)=\int_0^1 f(s) x(s) ds, \; \forall x\in X $ for some $f\in X-(0).$\\
We aim to show that under small valuers of $\epsilon>0$, this system is still exponentially stabilizable.\\
The  system $(S_1)$ can be reformulated as:	
\begin{eqnarray*}
	(S_2)	\begin{cases}
		\dot{x}(t)=\mathcal{A}x(t)+\nu(t) h(\zeta)x(t)\;\;& \mbox{in}\; (0,1)\times (0,\infty)\\
		x(0)=x_0\;\;& \mbox{in}\; (0,1)
	\end{cases}
\end{eqnarray*}	
where $\mathcal{A}: D(\mathcal{A})\subset X\to X$ is defined by:
$$\mathcal{A}x:= Ax-\epsilon hx,\;\forall x\in D(\mathcal{A}):=\left\lbrace x\in H^1(0,1), x(1)+\epsilon \psi(x)=0\right\rbrace.$$
We claim that $\mathcal{A}$ is the generator of a strongly continuous semigroup on $X.$
In order to verify this assertion, we will consider $\mathcal{A}$  as a perturbation of the generator $A$.
\\In order to write the system $(S_2)$ in the form (\ref{S1}), let us consider the function  $\theta(\zeta)= {\bf 1}(\zeta):=1,\; \zeta\in X$, which  is such that $
A_m\theta=0,$ and $ \theta(1)=1,$ where $A_m:=\frac{d}{d\zeta}$ with domain $D(A_m):=H^1(0,1)$.\\
Let us introduce the following operator
$$
B x =  h x - \psi (x) A_{-1}\theta,\;\; \forall x\in X
$$
which is a one to one  operator since we have  $\theta\not\in D(A).$\\
In the sequel, we will verify the assumptions of Theorem \ref{T1} and then conclude the stabilization of the perturbed system $(S_1)$.\\
$\bullet$ From the boundary conditions of $(S_2)$, we can see that
$$\forall x\in X,\;  x\in D(\mathcal{A}) \Leftrightarrow x\in H^1(0,1)\; \mbox{and} \;x + \epsilon \psi(x) \theta \in D(A)\cdot$$
\\
 This together with the definition of $\theta$ implies that for $x\in D(\mathcal{A}),$ we have
\begin{eqnarray*}
X \ni \mathcal{A} x &=&A_m x -\epsilon h x\\&=&A_m \big ( x +\epsilon \psi(x) \theta \big ) -\epsilon h x\\& =&A \big ( x +\epsilon \psi(x) \theta \big ) -\epsilon h x\\&=& A_{-1} \big ( x +\epsilon \psi(x) \theta \big ) -\epsilon h x\\&=&A_{-1} x -  \epsilon  B x\\&=& \big ( A_{-1}  -  \epsilon  B \big )|_X x
\end{eqnarray*}
Moreover, for all $ x\in D((A_{-1}  -  \epsilon  B)|_X),$ we have  $A_{-1} \big ( x +\epsilon \psi(x) \theta \big ) \in X$, i.e.  $ x +\epsilon \psi(x) \theta\in D(A)\subset H^1(0,1)$ which implies that $x\in H^1(0,1).$ Then we have
$( A_{-1}  -  \epsilon  B \big )|_X x =\mathcal{A} x.$ In other words,
$$\big (  \mathcal{A}, D(\mathcal{A})\big ) = \big (   ( A_{-1}  -  \epsilon  B  )|_X, D (A_{-1}  -  \epsilon  B  )|_X  ) \big ).$$
$\bullet$ The operator $(A_{-1}-\epsilon  B)_X$ is a generator if we can show that $$\int_{0}^{1}S_{-1}(1-r)\psi(u(r))A_{-1} \theta dr\in X,$$
or, equivalently
$$\int_{0}^{1}S_{-1}(1-r)\textbf{1}(.)\psi(u(r))dr\in D(A), \; \forall u\in L^2(0,1; X).
$$
We have
\begin{eqnarray*}
	\int_{0}^{1}S_{-1}(1-r)\textbf{1}(.)\psi(u(r))dr&=&\int_{0}^{1}\psi(u(r)) S(1-r)\textbf{1}({\bf \cdot})dr\\& = &\int_{\cdot}^{1} e^{-\alpha(1-r)}\psi(u(r))dr:=g(.).
\end{eqnarray*}
Since $\psi ou\in L^2(0,1),$ this implies that $g\in H^{1}(0,1)$ and $g(1)=0.$  In other words, $g\in D(A).$ Hence, for $\epsilon>0$ small enough, the system $(S_1)$ is well-posed.\\
$\bullet$  Here we can take $\mathfrak{X}_{-1}=span(A_{-1}\theta)$, so we obtain an admissible decomposition for the pair $(A_{-1},-\epsilon  B)$. Indeed, it is clear that $_{_X}\!\! Bx=h x,\; x\in X, $  so  $_{_X}\!\! B$ is  a bounded operator from $X$ to $X$.
\\
Moreover, for all $x\in D((A_{-1}-\epsilon  B)|_X) =D\big ((A_{-1}+\epsilon \psi(.) A_{-1}\theta)|_X \big ), $ we have $$
A_{-1} x=A_{-1} (x+ \epsilon \psi(x) \theta) -  \epsilon \psi(x) A_{-1}\theta= A (x+ \epsilon \psi(x) \theta)  -  \epsilon \psi(x) A_{-1} \theta,
$$
from which it comes that
$$_{_X}\!\!(A_{-1}) x=A (x+ \epsilon \psi(x) \theta), \; \forall  x\in D((A_{-1}-\epsilon  B)|_X),$$
where
$$
D((A_{-1}-\epsilon  B)|_X)=\{x\in  L^2(0,1)\: / \; x+ \epsilon \psi(x) \theta\in D(A)\}
$$
Then for $x\in D\big ( ( A_{-1}  -  \epsilon  B \big )|_X \big ), $ we have $ (A_{-1}-\epsilon  B)x \in X $ or equivalently $ x+ \epsilon \psi(x) \theta\in D(A),$ and
\begin{eqnarray*}
\langle_{_X}\!\!(A_{-1})x,x\rangle&=&\langle A(x+\epsilon \psi(x)\theta),x\rangle\\&=&\langle A_m(x+\epsilon \psi(x)\theta),x\rangle\\&=&\langle A_mx,x\rangle\\&=&\int_0^1x'(s)x(s) ds-\alpha\| x\|^2\\
&\le&\big (\frac{\epsilon^2 \|f\|^2}{2} -\alpha \big ) \|x\|^2 - \frac{1}{2} x^2(0)\cdot	
\end{eqnarray*}
Thus the operator  $_{_X}\!\!(A_{-1})$ is dissipative in $D((A_{-1}-\epsilon  B)|_X)$ for every $0<\epsilon\le \frac{\big ( 2 \alpha \big)^{1/2}}{\|f\|}$.
\\		
$\bullet$  Finally, the observation estimate follows from the fact that $h\ge c>0$ and that for any $x\in X, $ the mapping  $t\mapsto \|S(t)x\|$ is decreasing.\\
We conclude by Theorem \ref{T1} that for $\epsilon>0$ small enough,  the control $\nu(t)=-\epsilon {\bf 1}_{\{t\ge0:\; x(t)\ne0\}}$ guarantees  the exponentially stabilization of the system $(S_1)$.
\end{example}

\section{Conclusion}
In this paper we have shown that it is possible for  a linear system with dissipative dynamic, to be  exponentially stable under small  Desch-Schapacher  perturbations  of the dynamic. The main assumptions of sufficiency are formulated in terms of admissibility  and  observability assumptions  of unbounded linear operators. An explicit decay rate of the stabilized state is given.   The previous research on this problem concerned either bounded or Miyadera's type perturbations \cite{Liu,Ouz17}. The main stabilization  result is further applied to show the uniform exponential stabilization of  unbounded bilinear reaction diffusion and transport equations using a bang bang controller.


%

 \section*{Conflict of interest}

 The authors declare that they have no conflict of interest.



\end{document}